\documentclass[12pt]{article}
\usepackage[english]{babel}
\usepackage[latin1]{inputenc}
\usepackage[T1]{fontenc}
\usepackage{amsmath,amssymb,amsthm,xspace,enumerate,url}
\usepackage[hypertex]{hyperref}
\usepackage{svn}
\newcommand{\Ind}{
 \setbox0=\hbox{$x$}\kern\wd0\hbox to 0pt{\hss$
 \mid$\hss}\lower.9\ht0\hbox to 0pt{\hss$\smile$\hss}\kern\wd0
}

\newcommand{\Notind}{
 \setbox0=\hbox{$x$}\kern\wd0\hbox to 0pt{\mathchardef
 \nn=12854\hss$\nn$\kern1.4\wd0\hss}\hbox to 0pt{\hss$\mid$\hss}\lower.9\ht0
 \hbox to 0pt{\hss$\smile$\hss}\kern\wd0
}

\usepackage{srcltx}

\title{Strongly transitive multiple trees
}

\author{Katrin Tent
}
\date{\today}

\newtheorem{satz}{Theorem}[section]
\newtheorem{theorem}{Theorem}[section]

\newtheorem{lemma}[satz]{Lemma}
\newtheorem{proposition}[satz]{Proposition}
\newtheorem{corollary}[satz]{Corollary}
\newtheorem{definition}[satz]{Definition}
\newtheorem{remark}[satz]{Remark}

\newcommand{\nc}{\newcommand}
\nc{\sa}{semialgebraic\xspace}
\nc{\el}{elementary\xspace}
\nc{\low}{lower \el}
\nc{\inv}[1]{\frac{1}{#1}}
\nc{\G}{\Gamma}
\nc{\Np}{\N_{\scriptscriptstyle >0}}
\nc{\Z}{\mathbb{Z}}
\nc{\Q}{\mathbb{Q}}
\nc{\N}{\mathbb{N}}

\nc{\Rp}{\R_{\scriptscriptstyle >0}}
\nc{\C}{\mathbb{C}}
\nc{\T}{\mathbb{T}}
\nc{\F}{\ensuremath{\mathcal{F}}\xspace}
\nc{\K}{\mathcal{K}}
\nc{\A}{\mathcal{A}}
\nc{\B}{\mathcal{B}}

\nc{\U}{\mathbb{U}}
\nc{\fps}{free pseudospace\xspace}
\nc{\fpsn}{free pseudospace of dimension $n$\xspace}
\nc{\fpsk}{free pseudospace of dimension $n-1$\xspace}

\nc{\E}{\mathbb{E}}
\nc{\Epsilon}{{\Large $\epsilon$}} 
\nc{\ap}{approximable\xspace}
\nc{\e}{\mathrm{e}}
\nc{\ii}{\,\mathrm{i}}
\nc{\Es}{\E\setminus\R_{\scriptscriptstyle\leq0}}

\DeclareMathOperator{\Aut}{Aut}

\DeclareMathOperator{\dist}{dist}

\begin{document}

\maketitle
\begin{abstract}
We give an amalgamation construction of free multiple trees with a strongly transitive automorphism group. The construction shows
that any partial codistance function on a tuple of finite trees can
be extended to yield multiple trees.

\end{abstract}
\section{Introduction}

Multiple trees are a generalization of twin trees and buildings. While twin buildings arise from certain Kac-Moody groups, multiple trees were introduced by Tits and Ronan in order to deal with several valuations at the same time.

Since multiple trees are 
more rigid than twin trees, it is natural to ask whether as in the
case of spherical buildings of rank at least three a complete classification 
might be possible at least under the assumption that the automorphism group be strongly transitive. The construction given here shows that this is not the 
case. Furthermore, any codistance function on a tuple of finite trees can
be extended to yield multiple trees of infinite valency with a BN-pair.

 Free constructions of twin trees were given in \cite{Tfree}. However that construction started from generalized polygons and did not yield 
multiple trees: for more than two trees, the definition of multiple trees 
entails a certain regularity between sets of pairwise opposite vertices
which cannot hold in generalized $n$-gons for $n>6$. For this reason, the 
construction given here is completely different from the one for twin
trees given in  \cite{Tfree}. M\"uhlherr and Struyve informed me that they
have a construction for free multiple trees similar to the free construction
of polygons. However, their construction does not yield strong transitivity.

I thank Max Horn for pointing out an error in an earlier version of the paper.
\section{Construction}\label{sec:construction}

Recall that
 given an infinite tree $T$ without end-vertices, a codistance on $T$ is a mapping $d^*$ from $T$
 to the set $\N$ of nonnegative integers, such that, if $d^*(v)=n$ and $v'\sim v$ in $T$, then $d^*(v')\in\{n-1,n+1\}$. Moreover, if $n > 0$ then $d^*(v') = n+1$ for a unique $v'\sim v$ where we write $x\sim y$ if $x$ is a neighbour of $y$.
Given a (possibly infinite) family $\{T_i\}_{i\in I}$ of  trees, a multiple tree over $\{T_i\}_{i\in I}$ is defined by a codistance function $d\colon \prod_{i\in I}T_i\rightarrow\N$ such that, for any choice of  $k\in I$ and any $\bar a = (a_i)_{i\in I}\in\prod_{i\in I}T_i$, the function $d^*_{a;k}$ induced by $d$ on the graph $\{(x_h)_{h\in I} |$ $x_h = a_h$ for $h\neq k\}\cong T_k$ is a codistance on $T_k$. It will also be convenient to denote by $d^*_{\bar a;i,j}$ 
the codistance induced on the graph $\{(x_h)_{h\in I} |$ $x_h = a_h$ for $h\neq i,j\}\cong T_i\times T_j$

\begin{remark}\label{r:neighbour}\upshape
Suppose that $d^*(x_1,\ldots ,x_n)=k>0$ and that $y_i\sim x_i,y_j\sim x_j$ 
are such that $d^*_{\bar x;i,j}(y_i,x_j)=k+1=
d_{\bar x;i,j}^*(x_i,y_j)$. Then it follows easily that
we have $d_{\bar x;i,j}^*(y_i,y_j)=k+2$.
\end{remark}

\begin{definition}
Let $\K_n$ be the class of $n$-tuples of nonempty finite trees \[((A_1,\ldots A_n), d)\] with a codistance function $d:\prod_{i=1,\ldots n}A_i\to\N$ such that the following holds
\begin{enumerate}
\item there are $x_i\in A_i,i=1,\ldots, n$ with   $d^*(x_1,\ldots x_n)=0$;
\item if  $d^*(x_1,\ldots x_n)=k$, then for each $i=1,\ldots, n$
and any $y\sim x_i$ we have
 $d_{\bar x;i}^*(y)\in\{k+1,k-1\}$.
 
\item if  $d^*(x_1,\ldots x_n)=k>0$, then for each $i=1,\ldots, n$
there is at most one $y\sim x_i$ with
 $d_{\bar x;i}^*(y)=k+1$.
\item if $d^*(x_1,\ldots ,x_n)=k>0$ and  $y_i\sim x_i,y_j\sim x_j$ 
are such that $d^*_{\bar x;i,j}(y_i,x_j)=k+1=
d_{\bar x;i,j}^*(x_i,y_j)$, then $d_{\bar x;i,j}^*(y_i,y_j)=k+2$.
\end{enumerate}

\end{definition}

\begin{definition}
Let $\A=(\bar A,d^*)$ be in $\K_n$. We call $(x_1,\ldots x_n),(y_1,\ldots,  y_n)$ in $A$ geodesic if \[d^*(y_1,\ldots, y_n)=d^*(x_1,\ldots x_n)-D\]
where $D=\sum_{j\leq n}\dist(x_j,y_j)$  and $\dist$ denotes the graph theoretic
distance.
\end{definition}

\begin{lemma}\label{l:geodesic}
Let $\A=(\bar A,d)$ be in $\K_n$.
If  $(x_1,\ldots x_n)$ and $(y_1,\ldots, y_n)$ are geodesic, then  we have
\[d^*(y_1,\ldots, y_n)
 \leq  d^*(a_1,\ldots, a_n)
 \leq d^*(x_1,\ldots x_n)\]
if 
all $a_j\in A_j,j=1,\ldots n$ 
are on the geodesic from $y_j$ to $x_j$. If $d^*(y_1,\ldots, y_n)>0$ and $a_j\sim y_j$,
then \[d_{\bar y,j}^*(y_j)
 \leq  d_{\bar y,j}^*(a_1j)
 \leq d_{\bar y,j}^*(x_j)\]
if and only if
$a_j$ is on the geodesic from $y_j$ to $x_j$.
\end{lemma}

\begin{proof}
The first part follows immediately from the definition, the second part 
follows from Axiom 3.
\end{proof}

Let $\A=(\bar A,d)$ be in $\K_n$. We say that $d^*$ is \emph{locally maximal} in $(x_1,\ldots x_n)$ if there is no $y\sim x_i$ for any $i=1,\ldots n$ such that
\[d_{\bar x;i}^*(y)>d_{\bar x;i}^*(x_i).\]

\begin{definition}\label{d:good} Let $\A=(\bar A,d^*)$ be in $\K_n$. The following $1$-point extensions are called elementary good extensions of $\A$:
\begin{enumerate}
\item add a vertex $y$ to $A_k$ for some $1\leq k\leq n$ with $y\sim x_k\in A_k$ and for any $x_i\in A_i,i\neq k$ put
\[d_{\bar x;k}^*(y)=|d_{\bar x;k}^*(x_k)-1|.\]

\item if $d^*(x_1,\ldots, x_n)$ is locally maximal in $A$
 add a vertex $y$ to $A_k$ with $y\sim x_k$ and extend $d^*$ to the extension  as follows:

If $(y_1,\ldots, y_{k-1},x_k,y_{k+1},\ldots, y_n)$ is geodesic with $(x_1,\ldots, x_n)$, then

\[d_{\bar y;k}^*(y)=d_{\bar x;k}^*(x_k)+1.\]

and otherwise
\[d_{\bar y;k}^*(y)=|d_{\bar x;k}^*(x_k)-1|.\]

\end{enumerate}

For $A,B\in\K_n$ we say that $A$ is good in  $B$ if $B$ arises from $A$ by a finite
sequence of elementary good extensions.

\end{definition}

\begin{lemma}\label{l:goodext}
Suppose that $(A,d^*)\in\K_n$ and $(A\cup\{y\},d^*)$ is an elementary good extension.
Then $(A\cup\{y\},d^*)\in\K_n$.
\end{lemma}

\begin{proof}
If the extension is of type 1. all conditions continue to hold automatically.
We have to show that the conditions hold for extensions of type 2.
Conditions 1. and 2. are still clear.

For 3. we have to
show that for any $(a_1,\ldots a_n)\in A\cup\{y\}$ with $d^*(a_1,\ldots a_n)>0$ and $1\leq i\leq n$ there
is at most one $x\sim a_i$ with 
$d_{\bar a;i}^*(x)>d_{\bar a;i}^*(a_i)$. Suppose
that $d^*(x_1,\ldots, x_n)$ is locally maximal in $A$ and a vertex 
$y$ was attached to $A_k$ with $y\sim x_k$. If $a_k\neq x_k,y$, then
all vertices are inside $A$ and since $d^*$ was not changed on $A$, the claim
remains true.

So suppose $a_k=x_k$. 
If $\bar a=\bar x$, then the claim follows from the local maximality.
Hence we may assume that for some $1\leq j\leq n$ we have $a_j\neq x_j$
(and clearly $j\neq k$).
Since there is a unique new vertex $y$, we only
have to consider the case  $i=k$ and $d_{\bar a;k}^*(y)>d_{\bar a;k}^*(a_i)$. 
Then 
$(a_1,\ldots a_n)$ is geodesic with 
$(x_1,\ldots, x_n)$. Thus  for $b\sim a_j$ 
in  the interval $[a_j,x_j]$ we have \[d_{\bar a;i}^*(b)=d^*(a_i)+1.\]
Suppose that for some further $x\sim x_k=a_k,x\neq y$ we have \[d_{\bar a;k}^*(x)=d_{\bar a;k}^*(a_k)+1.\]

Since $A\in\K_n$ this implies
\[d_{\bar a;k,j}^*(x,b)=d_{\bar a;k,j}^*(x_k,a_j)+2\]

and hence
\[d_{\bar a;k,j}^*(x_k,b)\geq d_{\bar a;k,j}^*(x_k,a_j)+1\]
contradicting the locally maximal choice of $d^*(x_1,\ldots, x_n)$ (remember that $a_k=x_k$).

Finally consider the case $a_k=y$.
For  $i=k$ there is nothing to show since $y$ has a unique neighbour,
so suppose $i\neq k$.
If \[d_{\bar a;k}^*(y)>d_{\bar a;k}^*(x_k)\]
then $(a_1,\ldots,a_{k-1},x_k,a_{k+1},\ldots, a_n)$ is geodesic with
$(x_1,\ldots, x_n)$. 

If $x\sim a_i$  is such that
\[d_{\bar a;i}^*(x)=d_{\bar a;i,k}^*(x,y)>d_{\bar a;i}^*(a_i)=d_{\bar a;k}^*(y)>d_{\bar a;k}^*(x_k)\]

then \[d_{\bar a;i,k}^*(x,y)>d_{\bar a;i,k}^*(a_i,y)=d_{\bar a;k}^*(y)>d_{\bar a;i,k}^*(a_i,x_k)\]

and therefore
\[d_{\bar a;i,k}^*(x,y)>d_{\bar a;i,k}^*(x,x_k).\]

Hence by Lemma~\ref{l:geodesic} $x$ is the unique neighbour of $a_i$ closer to $x_i$.

Next suppose \[d_{\bar a;k}^*(y)<
d_{\bar a;k}^*(x_k)\]
and that there are $b_1,b_2\sim a_i$ with \[d_{\bar a;i}^*(b_1)=d_{\bar a;i}^*(b_2)>
d_{\bar a;i}^*(a_i).\]

If
\[d_{\bar a;i,k}^*(b_1,x_k)=d_{\bar a;i,k}^*(b_2,x_k)<
d_{\bar a;i,k}^*(a_i,x_k)=d_{\bar a;k}^*(x_k)\]
then 
\[d_{\bar a;i}^*(y)=d_{\bar a;i,k}^*(b_1,x_k)=d_{\bar a;i,k}^*(b_2,x_k)<
d_{\bar a;i,k}^*(b_1,y)=d_{\bar a;i,k}^*(b_2,y).\]

Then $(a_1.\ldots a_{i-1},b_s,a_{i+1},\ldots, a_{k-1},x_k,a_{k+1},\ldots ,a_n)$
is geodesic with $(x_1,\ldots, x_n)$ for $s=1,2$. But this clearly implies
that also \[(a_1,\ldots, a_{i-1},a_i,a_{i+1},\ldots, a_{k-1},x_k,a_{k+1},\ldots ,a_n)\] is geodesic with $(x_1,\ldots, x_n)$, contradicting
\[d_{\bar a;k}^*(y)< d_{\bar a;k}^*(x_k)\]

As $A$ is in $\K_n$ we may therefore assume
\[ d_{\bar a;i,k}^*(b_1,x_k)
< d_{\bar a;i,k}^*(a_i,x_k)
<  d_{\bar a;i,k}^*(b_2,x_k).\]

Then
\[ d_{\bar a;i,k}^*(b_1,x_k)
= d_{\bar a;i,k}^*(a_i,y)< d_{\bar a;i,k}^*(b_1,y)= 
d_{\bar a;i,k}^*(b_2,y)< d_{\bar a;i,k}^*(b_2,x_k).\]

This implies that $(a_1.\ldots a_{i-1},b_2,a_{i+1},\ldots, a_{k-1},x_k,a_{k+1},\ldots ,a_n)$
is geodesic with $(x_1,\ldots, x_n)$ from which we again conclude that
also
\[(a_1.\ldots a_{i-1},a_i,a_{i+1},\ldots, a_{k-1},x_k,a_{k+1},\ldots ,a_n)\]
is geodesic with $(x_1,\ldots, x_n)$, a contradiction.

It remains to prove that Condition 4. holds:
if $d^*(a_1,\ldots ,a_n)=m>0$ and  $y_i\sim a_i,y_j\sim a_j$ 
are such that $d^*_{\bar a;i,j}(y_i,a_j)=m+1=
d_{\bar a;i,j}^*(a_i,y_j)$, then $d_{\bar a;i,j}^*(y_i,y_j)=m+2$.

Since $A\in\K_n$ we only have to check the situation when $a_k=y$ or when
$a_k=x_k$ and $i=k$. 

First assume $a_k=x_k$ and $i=k$. Assume  $y\sim x_k,y_j\sim a_j$ 
are such that
$d^*_{\bar a;k,j}(y,a_j)=m+1=d_{\bar a;k,j}^*(x_k,y_j)$.  Then
$(a_1,\ldots a_n)$
is geodesic with $(x_1,\ldots x_n)$. Lemma~\ref{l:geodesic} implies
that $y_j\sim a_j$ lies between $a_j$ and $x_j$ and hence
$(a_1,\ldots, a_{j-1},y_j,a_{j+1},\ldots, a_n)$
is geodesic with $(x_1,\ldots x_n)$. Hence $d_{\bar a;k,j}^*(y,y_j)=m+2$.

Next we assume $a_k=y$ and $i=k$. Assume  $y_k=x_k\sim y,y_j\sim a_j$ 
are such that
$d^*_{\bar a;k,j}(x_k,a_j)=m+1=d_{\bar a;k,j}^*(y,y_j)$. 
If $d^*_{\bar a;k,j}(x_k,y_j)<d_{\bar a;k,j}^*(y,y_j)$, then
$(a_1,\ldots, a_{k-1},x_k,a_{k+1},\ldots, y_j,a_{j+1},\ldots a_n)$
is geodesic with $(x_1,\ldots x_n)$. Now $d^*_{\bar a;k,j}(x_k,a_j)=m+1$
implies that also $(a_1,\ldots, a_{k-1},x_k,a_{k+1},\ldots, a_j,a_{j+1},\ldots a_n)$ is geodesic with $(x_1,\ldots x_n)$, and hence $d^*_{\bar a;k,j}(x_k,a_j)<
d^*_{\bar a;k,j}(y,a_j)$,
 a contradiction.

Now assume $a_k=y$ and $i,j\neq k$. Let $y_i\sim a_i,y_j\sim a_j$ 
be such that
$d^*_{\bar a;i,j}(y_i,a_j)=m+1=d_{\bar a;i,j}^*(a_i,y_j)$. 
If $d^*_{\bar a;i,j}(y_i,y_j)<d_{\bar a;i,j}^*(a_i,y_j)$, then
$(a_1,\ldots, y_j,a_{j+1},\ldots a_n)$ and hence $(a_1,\ldots, a_n)$
are geodesic with $(x_1,\ldots x_n)$ (while
$(a_1,\ldots, y_i,a_{i+1},\ldots a_n)$ is not geodesic with $(x_1,\ldots x_n)$.) Thus for the unique $b_i\sim a_i$
between $a_i$ and $x_i$ we have $d^*_{\bar a;i,j}(b_i,y_j)=m+2$.
But since $A\cup\{y\}$ satisfies Condition 2. we must have $y_i=b_i$, a contradiction.
\end{proof}

\begin{lemma}\label{l:goodemb}
Suppose that $(A,d^*),(A\cup\{y\},d^*)\in\K_n$. Then $A\cup\{y\}$ is an elementary good extension of $A$.
\end{lemma}
\begin{proof}
Let $y\sim x_k$ for some (unique) $x_k\in A_k$. If 
 $d_{\bar x;k}^*(y)<d_{\bar x;k}^*(x_k)$ for all $(x_1,\ldots, x_n)\in A$, then $A\cup\{y\}$ is an extension of type 1.
So suppose for some  $(x_1,\ldots, x_n)\in A$ 
we have $d_{\bar x;k}^*(y)>d_{\bar x;k}^*(x_k)$, which we may assume to be locally  maximal. Let $(z_1,\ldots, z_n)\in A$ with $z_k=x_k$  and $d_{\bar z;k}^*(y)>d_{\bar z;k}^*(x_k)$. We have to show that $\bar z$ and $\bar x$ are geodesic.
Otherwise for some $j$ on the path $\gamma=(x_j=y_0,\ldots, y_m=z_j)$ there is some $y_s$ such that 
\[d_{\bar x;j}^*(y_s)<d_{\bar x;j}^*(y_{s-1},d_{\bar x;j}^*(y_{s+1}\]
contradicting Axiom 3.
\end{proof}

\begin{corollary}\label{c:goodemb}
Suppose that $(A,d^*),(B,d^*)\in\K_n$ with $A\subseteq B$. Then $A$ is good in $B$.\qed
\end{corollary}

\begin{lemma} The class $\K_n$ has the amalgamation property.
\end{lemma}

\begin{proof}
 Clearly it suffices inductively to prove this for $A\subseteq B,C$
where $B,C$ are elementary good extensions of $A$, so $B=A\cup\{b\},C=A\cup\{c\}, b\in B_i, c\in C_j$ with unique neighbours $b',c'$. 

If $i=j$ and at least one of the extensions is of type 1., then $A\cup \{b,c\}$
with the codistance induced by $B,C$ is in $\K_n$. 
So suppose both extensions are of type 2. In this case $A\cup \{b,c\}$
with the codistance induced by $B,C$ is in $\K_n$ unless $b'=c'$ and $b'$ and $c'$ increase the distance to the same locally maximal tuple. In this case we identify $b$ and $c$ and choose $B$ as the amalgam.

Now suppose $i<j$.
If $B$ is an elementary good extension of $A$ of the first
kind, define $d^*$ on $A\cup \{b,c\}$ by 
\[d_{\bar x;i,j}^*(b,c))
=  d_{\bar x;i,j}^*(b',c)-1.\]
Then $A\cup \{b,c\}$ is an elementary good extension of $C$ of the first
kind and is therefore in $\K_n$ by Lemma~\ref{l:goodext}.

Next suppose that $B$ and $C$ are elementary good extensions of $A$ of the second kind and that $b,c$ were attached to different locally maximal
tuples. Let $(y_1,\ldots, y_{i-1},b',y_{i+1},\ldots y_n)$ be the locally maximal
tuples to which $b$ was attached. If $(x_1,\ldots, x_{i-1},b',x_{i+1},\ldots, x_{j-1},c,x_{j+1},\ldots, x_n)$ is geodesic with 
$(y_1,\ldots, y_{i-1},b',y_{i+1},\ldots y_n)$  define $d^*$ on $A\cup \{b,c\}$ by 
\[d_{\bar x;i,j}^*(b,c)= d_{\bar x;i,j}^*(b',c)+1.
\] 
and otherwise 
\[d_{\bar y;i,j}^*(b,c)
= |d_{\bar y;i,j}^*(b',c)-1|.\]

Then $A\cup \{b,c\}$ is an elementary good extension of $C$ of the second
kind and is therefore in $\K_n$ by Lemma~\ref{l:goodext}.

Finally suppose that $b,c$ are attached to the same locally maximal
tuple $(x_1,\ldots,x_{i-1},b',x_{i+1},\ldots,x_{j-1},c',x_{j+1},\ldots, x_n)$.

If $(y_1,\ldots, y_{i-1},b',y_{i+1},\ldots,y_{j-1},c',y_{j+1},\ldots y_n)$ is geodesic with 
\[(x_1,\ldots,x_{i-1},b',x_{i+1},\ldots,x_{j-1},c',x_{j+1},\ldots, x_n),\] then
\[d_{\bar y;i,j}^*(b,c) =  d_{\bar x;i,j}^*(b,c')+1.\]

and otherwise 
\[d_{\bar y;i,j}^*(b,c)
= |d_{\bar y;i,j}^*(b,c')-1|.\]
Then $A\cup \{b,c\}$ is an elementary good extension of $C$ of the second
kind and is therefore in $\K_n$ by Lemma~\ref{l:goodext}.

\end{proof}

Using a bit of model theory, we may extend the language 
of graphs with $n$-ary predicates $d^*_k,k\in\N$ denoting the codistance and
by binary function symbols $f_i(x,y),i\in\N$ with $f_i(x,y)=z$ if $z$ is the
$i^{th}$ element on the path from $x$ to $y$ if such a path of length at least $i$ exists and $z=x$ otherwise. In this expanded language, the substructure generated by a subset $A$ will include all finite geodesic paths between elements from the same coordinate tree $A_k$. Hence in this language, the class $\K_n$ is closed
under finitely generated substructures. Since by Axiom 1. any $A\in\K_n$ contains the structure consisting only of an $n$-tuple $(x_1,\ldots, x_n)$ with $d^*(\bar x)=0$, the class $\K_n$ has a Fra\"iss\'e limit $M_n$ (see e.g. \cite{TZ}, Ch. 4.4), i.e. a countable structure whose finite substructures are exactly the structures isomorphic to elements of $\K_n$ and whose
automorphism group acts  transitively on isomorphism classes of finite substructures.

Suppose $A$ is an $n$-fold tree and $(x_1,\ldots, x_n)\in A$ is such that $d^*(\bar x)=0$.
and $y,z\sim x_i$ for some $1\leq i\leq n$. Then the set \[\T_{\bar x,y}=\bigcup_{1\leq j\leq n} \{z\in A_j
\colon d_{\bar x;i,j}^*(y,z)>d_{\bar x;i,j}^*(x_i,z) \}\] is called a \emph{half apartment}, and
 the set \[\T_{\bar x;y,z}=\T_{\bar x;y}\cup \T_{\bar x;z}\] is called an \emph{apartment}. 

The automorphism group of a multiple tree $A=\{A_i\}_{i\leq n}$ is said to act \emph{strongly transitively} if it acts transitively on the set of \emph{marked apartments},
i.e. transitively on the set $\{(x_1,\ldots, x_n,y)\colon x_i\in A_i, y\sim x_1\}$. This is equivalent to the automorphism group having a BN-pair (see \cite{AB} Ch. 6).

Hence we obtain our main theorem:

\begin{theorem}\label{t:multipletree}
The Fra\"iss\'e limit $M_n$ is a multiple tree whose automorphism group acts strongly transitively on the set of apartments.
\end{theorem}

\begin{proof}
It is clear from the construction that if $d_{\bar x;i}^*(x)=k$ then there exists a unique $y\sim x_i$ such that $d_{\bar x;i}^*(y) = k+1$. 

The strong transitivity of $\Aut(M_n)$ is immediate by the properties of the
Fra\"iss\'e limit:  the structure $\{(x_1,\ldots, x_n)\}\cup\{ y\}$ with $y\sim x_1$ is in $\K_n$ and $\Aut(M_n)$ acts transitively on the set of substructures
of $M_n$ isomorphic to it.
\end{proof}

For a half apartment $\T_{\bar x;y}$ the corresponding \emph{root group} $U_{\T,\bar x;y}$ in $\Aut(A)$ is the subgroup  of $\Aut(A)$ fixing  the set $\T_{\bar x;y}\cup \{v\colon v\sim z \mbox{  for some } z\in \T_{\bar x;y}\}$ pointwise.

We say that a multiple tree satisfies the \emph{Moufang condition} if for each root  the corresponding root group acts transitively on the set of apartments containing the given root.

\begin{remark}\label{r:regular}
It is easy to see that if a root group acts transitively 
on the set of apartments containing the given  half-apartment, then it
acts regularly. In other words, the stabilizer of an apartment inside a root group is trivial (see \cite{ronan}, Sec. 4).
\end{remark}

\begin{proposition}\label{p:moufang}
The automorphism group of the multiple tree $M_n$ does not satisfy the Moufang condition. 
\end{proposition}

\begin{proof}
We claim that in fact all root groups are trivial. Note that by the properties of the Fra\"iss\'e limit, if $(x_1,\ldots, x_n)$ in $M_n$ is such that $d^*(x_1,\ldots, x_n)=0$, the stabilizer $H$ of $x_1,\ldots x_n$ acts highly transitively (i.e. $m$-transitively for any $m$) on the set of vertices  $z\sim x_i$ since for any $m\in\N$
the structure $\{x_1,\ldots, x_n, z_1,\ldots z_m\}$ with $z_j\sim x_i, j=1,\ldots, m$ is in $\K_n$ and $d^*$ is uniquely determined. In particular, the stabilizer $H_y$ of $y$ in $H$ acts highly transitively on the set of $z\sim x_i,z\neq y$ and normalizes
$U_{\T,\bar x;y}$. If $U_{\T,\bar x;y}$ was nontrivial, this would contradict Remark~\ref{r:regular}, whence the claim. 

\end{proof}

\vspace{.5cm}
\noindent\parbox[t]{15em}{
Katrin Tent,\\
Mathematisches Institut,\\
Universit\"at M\"unster,\\
Einsteinstrasse 62,\\
D-48149 M\"unster,\\
Germany,\\
{\tt tent@math.uni-muenster.de}}


\vspace{1cm}


\begin{thebibliography}{WWW}


\bibitem{AB}
P. Abramenko, K. Brown, 
Buildings. Theory and Applications,
Graduate Texts in Mathematics, Vol. 248, Springer (2008)



\bibitem{ronantitsI}

M. Ronan, J. Tits,(1994) Twin trees I. Inventiones Mathematicae (1994), no. 116 463 -- 479. 

\bibitem{ronantitsII}

M. Ronan, J. Tits, Twin trees II: local structure and a universal construction. Israel Journal of Mathematics (1999),  no. 109, 349 -- 377.

\bibitem{ronan}
M. Ronan,
Multiple trees.
J. Algebra 271 (2004), no. 2, 673 -- 697. 

\bibitem{Tfree}
On polygons, twin trees and CAT(1)-spaces, in: Pure and Applied Mathematics Quarterly Volume 7, Number 3 (Special Issue: In honor of Professor Jacques Tits) 1023 -- 1038, 2011.
 
\bibitem{TZ}
  K. Tent, M. Ziegler, A Course in Model theory, to appear in ASL Lecture Notes in Logic, Cambridge University Press, 2012.


\end{thebibliography}
\end{document}